\documentclass[reqno,12pt]{amsart}
\usepackage[utf8]{inputenc}
\usepackage[left=1.2in, right=1.2in, top=1in]{geometry}

\usepackage{amsmath, amssymb,amsthm, mathrsfs,color,times,textcomp,verbatim}
\usepackage{xcolor}
\usepackage[colorlinks=true]{hyperref}
\hypersetup{urlcolor=blue, citecolor=red, linkcolor=blue}
\usepackage[square,numbers]{natbib}
\usepackage{pgfplots}
\usepackage{tikz}
\numberwithin{equation}{section}
\theoremstyle{plain}
\newtheorem{theorem}{Theorem}[section]
\newtheorem{proposition}[theorem]{Proposition}

\newtheorem{lemma}[theorem]{Lemma}

\usepackage{graphicx}

\theoremstyle{definition}

\newtheorem{remark}[theorem]{Remark}
\title{Liouville type theorem for a class quasilinear $p$-Laplace type equation on the sphere}
\begin{document}
\author{Daowen Lin}
\address{School of Mathematical Sciences, University of Science and Technology of China, Hefei, Anhui Province, P. R. China, 230006}
\email{lindw@mail.ustc.edu.cn}
\author{Xi-Nan Ma}
\address{School of Mathematical Sciences, University of Science and Technology of China, Hefei, Anhui Province, P. R. China, 230006}
\email{xinan@ustc.edu.cn}
\maketitle
\begin{abstract}
	We use the integral by parts  to get a Liouville type theorem for a class quasilinear $p$-Laplace type equation on the sphere, this $p$-Laplace type  equation  arises from the study of asymptotic behaviour near  the origin  for the semilinear $p$-Laplace  equation on the puncture ball  $B_1(o) \subset \mathbb{R}^n$. This  give a positive answer to L. V\'{e}ron's question in a paper \cite{Veron92}  and his book \cite{Veron} at page 440. 
\end{abstract}
\section{Introduction}
In Gidas-Spruck\cite{GS81}, they studies the Liouville type theorem for the nonnegative solution on the following semilinear elliptic  equation
\begin{equation}\label{1.1a}
	\Delta u+ u^q=0\,\,\,\,\,\textrm{in}\,\qquad \mathbb{R}^{n+1}
\end{equation} 
in the range of $1 < q < 2^* - 1$ where $2^* = \frac{2(n+1)}{n-1}$, and they obtained the unique solution is trivial solution.   Gidas-Spruck\cite{GS81}  proved their results via the method of vector fields and integral by parts motivated by Obata identity \cite{Obata71}.  

In order to study the asymptotic behaviour near  the origin  for the above  equation \eqref{1.1a} on the puncture ball  $B_1(o)\setminus \{o\} \subset R^{n+1}$,  their studied the following equation on sphere $S^n$. 
\begin{equation}\label{1.2}
	\Delta u+ u^q- \lambda u=0\,\,\,\,\,\textrm{in}\,\qquad \mathbb{S}^{n},
\end{equation} 
and their also got a Liouville type theorem under certain condition on $q, \lambda $. In a more genaral setting, Veron-Veron \cite{Veron91} got the  following theorem.
\begin{theorem}(V\'{e}ron-V\'{e}ron \cite{Veron91} )\label{Thm0.1}
	Assume $(M,g)$ is a compact Riemannian manifold without boundary of dimension $n\ge 2$, $	\Delta $ is the Laplace-Beltrami operator on $M$, $q>1, \lambda>0$ and $u$ 
	is a positive solution of 
	\begin{equation}\label{1.3}
		\Delta u+ u^q- \lambda u=0\,\,\,\,\,\textrm{on}\,  \qquad\mathbb{M}^{n}.
	\end{equation} 
	Assume also that the spestrum $\sigma(R(x))$ of the Ricci tensor $R$ of the metric $g$ satisfies 
	\begin{equation}\label{1.4}
	{\inf}_{x\in M}\min \sigma(R(x)) \ge \frac{n-1}{n}(q-1)\lambda,\\\quad 
	q\le \frac{n+2}{n-2}.
	\end{equation} 
	Moreover, assume that one of the two inequalities is strict if $(M,g)$
is conformally diffeomorphim to $(\mathbb{S}^n,g)$. Then $u$ is constant with the value ${\lambda}^\frac{1}{q-1}$.	
\end{theorem}

For  the following semilinear $p$-Laplace  equation
\begin{equation}\label{1.5}
	\Delta_p u+ u^q=0\,\,\,\,\,\textrm{in}\,\qquad \mathbb{R}^{n+1}.
\end{equation} 
In Serrin-Zou\cite{SZ02}, they got a Liouville type theorem for the nonnegative solution of equation \eqref{1.5}, 
in the range of  $1<p<n+1$ and $p-1 < q < p^* - 1$ where $p^* = \frac{(n+1)p}{n+1-p}$,  they got the unique solution is trivial solution.

In order to study the asymptotic behaviour near  the origin  for the above  equation \eqref{1.5} on the puncture ball  $B_1(o)\setminus \{o\} \subset \mathbb{R}^{n+1}$. As in Gidas-Spruck\cite{GS81}, V\'{e}ron \cite{Veron}  made the following observation. With the spherical coordinate $(r,\sigma)$, separable solutions of \eqref{1.5} under the form $u(x)=u(r,\sigma)=r^{-\alpha}\omega(\sigma)$ exist, then  $\omega$ satisfies
\begin{equation}\label{1.6}
	div((\alpha_{p,q}^2\omega^2+|\nabla \omega|^2)^{\frac{p-2}{2}}\nabla \omega)+|\omega|^{q-1}\omega-\lambda_{p,q}(\alpha^2_{p,q}\omega^2+|\nabla \omega|^2)^{\frac{p-2}{2}}\omega=0\quad \textrm{on}\,\qquad\mathbb{S}^n,
\end{equation}
where $$\lambda_{p,q}=\alpha_{p,q}(n+1-\alpha_{p,q}q), \qquad \alpha=\alpha_{p,q}=\frac{p}{q+1-p},$$ $div$ and $\nabla$ are operators  under the canonical metric on $\mathbb{S}^n$.

If $\lambda_{p,q}<0$, i.e. $p-1<q\leq\frac{(n+1)(p-1)}{n+1-p}$; integrating equations \eqref{1.6} shows that there exists no nontrivial solution to \eqref{1.6}.

For $q=\frac{np-n+p}{n-p}$, which is the Sobolev critical exponent, it is well known  that \eqref{1.5} admits nonconstant solutions. 

For $\lambda_{p,q}>0$, $\frac{(n+1)(p-1)}{n+1-p}<q<\frac{np-n+p}{n-p},$  in a paper \cite{Veron92} and his book\cite{Veron} at page 440,  L. V\'{e}ron asked if all \textbf{positive} solutions of \eqref{1.5} are constant. In this paper we confirm it:
\begin{theorem}\label{Them1}
 For $1<p<n$ and $\frac{(n+1)(p-1)}{n+1-p}<q<\frac{np-n+p}{n-p}$ with $\alpha_{p,q}=\frac{p}{q+1-p}$ and $\lambda_{p,q}=\alpha_{p,q}(n+1-\alpha_{p,q}q)$, any positive solution to \eqref{1.6} is constant.
\end{theorem}
\begin{remark}
	Using the similar computation, our proof also work on the closed Riemannian manifold $(M,g)$ with $Ric_{ij}\geq(n-1)g_{ij}$.
\end{remark}

 We give some reviews on the this related subject. Based on the technique developed in Dolbeault-Esteban-Loss\cite{DEL14},  Dolbeault-Esteban-Loss\cite{DEL16} finally solve the famous problem of the characterization of the optimal symmetry breaking region in Caffarelli-Kohn-Nirenberg inequalities\cite{CKN84} completely. As in Gidas-Spruck \cite{GS81},  Ma-Ou \cite{MO} get similar Liouville results  on Heisenberg group $\mathbb{H}^n$.  
 Inspired by Serrin-Zou\cite{SZ02}, Ciraolo-Figalli-Roncoroni\cite{CFR20} classify the positive energy finite solutions to \eqref{1.5} when $q=\frac{(n+1)p}{n+1-p}-1$ in convex cones with the help of some  prior estimates. For the $p$ version Caffareli-Kohn-Nirenberg inequalities, there are some partial results in Ciraolo-Corso\cite{CC22}. 
See the recent result for the critical $p$-Laplace equation in $\mathbb{R}^n$ by Q.Ou \cite{Ou22}.   In Ciraolo-Figalli-Roncoroni\cite{CFR20}, 
 an important Lemma  by \cite{AKM18} and \cite{CM18} for the research of $p$-Laplacian equations has been used.

In the above  $p$-Laplace equation papers, the authors  always intruduce only one parameter to get their result.
 In the proof of our theorem \ref{Them1}, we introduce three parameters and use the  Lemma from \cite{AKM18} or \cite{CM18} to complete our proof.

The paper is organized as follows. In section 2, we introduce some notations and prove an integral equality. Then we use the integral equality through choosing  these parameters to prove the Theorem \ref{Them1} in section 3.

\section{An Integral Equality }
In this section ,we drive a useful equality.

Let $\omega=v^{-\beta}$, $\beta\neq0$. We denote $k=(\beta+1)(p-1)-\beta q$, $Q=(\alpha^2v^2+\beta^2|\nabla v|^2)^{\frac{1}{2}}$, $X^i=Q^{p-2}v_i$, $X^{i}_{j}=(Q^{p-2}v_i)_j$, $E^{i}_j=X^i_j-\frac{div_g(X^i)}{n}g_{ij}$, $L^{i}_j=Q^{p-2}(\frac{v_iv_j}{v}-\frac{|\nabla v|^2}{nv}g_{ij})$. Then
\begin{align}\label{2.1a}
 E^i_jE^j_i=X^i_jX^j_i-\frac{X^i_iX^j_j}{n}.
\end{align}
and 
\begin{align}\label{2.1b}
	|L^{i}_j|^2=\frac{n-1}{n}Q^{2p-4}v^{-2}|\nabla v|^4
\end{align}

We modify $E^i_j$ to deal with the subcritical case. We set $F^i_j=E^{i}_j+\varepsilon div_g(X^i)g_{ij}$, for some $\varepsilon\neq 0$, which is  $ F^i_j=X^{i}_j+(\varepsilon-\frac{1}{n})div_g(X^i)g_{ij}$. Using the fact that $L^i_j$ is trace free, we have
\begin{align}\label{2.1}
	F^i_jF^j_i&=X^i_jX^j_i+(n\varepsilon^2-\frac{1}{n})(X^l_l)^2,\\
	F^i_jL^i_j&=E^i_jL^i_j=Q^{p-2}\left[(Q^{p-2}v_i)_j\frac{v_iv_j}{v}-(Q^{p-2}v_l)_l\frac{|\nabla v|^2}{nv}\right].\label{2.2}
\end{align}
Then our equation \eqref{1.2} becomes
\begin{align}\label{2.3}
	X^i_i-(\beta+1)(p-1)v^{-1}|\nabla v|^2Q^{p-2}-\beta^{-1}v^k+\beta^{-1}\lambda Q^{p-2}v=0.
\end{align} 
Multiplying \eqref{2.3} with $v^aX^j_j$ and integrating on $\mathbb{S}^n$, $a$ shall  be determined  in later, we get
\begin{equation}
\begin{aligned}\label{2.4}
	&\int v^aX^i_iX^j_j-(\beta+1)(p-1)\int v^{a-1}|\nabla v|^2Q^{p-2}X^j_j-\beta^{-1}\int v^{k+a}X^j_j\\+&\beta^{-1}\lambda\int v^{a+1}Q^{p-2}X^j_j=0.
\end{aligned}
\end{equation}
As by integral by parts, for the third term in \eqref{2.4} we have  
\begin{align}\label{2.5}
	-\beta^{-1}\int v^{k+a}X^j_j=\beta^{-1}(k+a)\int v^{k+a-1}|\nabla v|^2Q^{p-2}.
\end{align}
Note that 
\begin{align}\label{2.6}
	(Q^{p-2})_j=\left[(\alpha^2v^2+\beta^2|\nabla v|^2)^{\frac{p-2}{2}}\right]_j=(p-2)Q^{p-4}(\alpha^2vv_j+\beta^2v_lv_{lj}).
\end{align}
Now we set $f=vv_jv_iv_{ji}-|\nabla v|^4$, then the last term in \eqref{2.4} becomes 
\begin{equation} 
\begin{aligned}\label{2.7}
	&\beta^{-1}\lambda\int v^{a+1}Q^{p-2}(Q^{p-2}v_j)_j\\=&-\beta^{-1}\lambda(a+1)\int v^a|\nabla v|^2Q^{2p-4}-\beta^{-1}\lambda\int v^{a+1}(Q^{p-2})_jQ^{p-2}v_j\\=&-\beta^{-1}\lambda(a+p-1)\int v^a|\nabla v|^2Q^{2p-4}-(p-2)\beta\lambda\int v^aQ^{2p-6}f.
\end{aligned}
\end{equation}
As for the first term  in \eqref{2.4},  we observe that $(Q^{p-2}v_j)_{ji}=(Q^{p-2}v_j)_{ij}-R_{ji}v_jQ^{p-2}$ where $R_{ij}$ is the Ricci curvature.  So we  have
\begin{align*}
	\int v^a(Q^{p-2}v_i)_iX^j_j=&-a\int v^{a-1}Q^{p-2}|\nabla v|^2X^j_j-\int v^aQ^{p-2}v_iX^j_{ji}\\=&-a\int v^{a-1}Q^{p-2}|\nabla v|^2X^j_j-\int v^a(Q^{p-2}v_j)_ijQ^{p-2}v_i+\int v^aQ^{2p-4}R_{ji}v_jv_i\\=&-a\int v^{a-1}Q^{p-2}|\nabla v|^2X^j_j+\int v^a(Q^{p-2}v_i)_j(Q^{p-2}v_j)_i\\&+a\int v^{a-1}(Q^{p-2}v_j)_iQ^{p-2}v_iv_j+\int v^aQ^{2p-4}R_{ji}v_jv_i.
\end{align*}
Invoking  \eqref{2.1}, it follows that  the first term  in \eqref{2.4} becomes
\begin{equation}\label{2.8}
	\begin{aligned}
		\int v^aX^i_iX^j_j=&-\frac{na}{n-1+n^2\varepsilon^2}\int v^{a-1}Q^{p-2}|\nabla v|^2X^i_i+\frac{n}{n-1+n^2\varepsilon^2}\int v^aR_{ji}v_jv_i\\&+\frac{na}{n-1+n^2\varepsilon^2}\int v^{a-1}(Q^{p-2}v_j)_iQ^{p-2}v_iv_j+\frac{n}{n-1+n^2\varepsilon^2}\int v^aF^i_jF^j_i.
	\end{aligned}
\end{equation}

To deal with the term in \eqref{2.5}, we times the equation \eqref{2.3} with $|\nabla v|^2v^{a-1}Q^{p-2}$. Then for the third term in \eqref{2.4}, we get
\begin{align}\label{2.9}
	&-\beta^{-1}\int v^{k+a}X^j_j=\beta^{-1}(k+a)\int v^{a+k-1}|\nabla v|^2Q^{p-2}
	\\
	=&(k+a)\int v^{a-1}|\nabla v|^2Q^{p-2}X^i_i-(k+a)(\beta+1)(p-1)\int v^{a-2}|\nabla v|^4Q^{2p-4}\nonumber\\&+(k+a)\beta^{-1}\lambda\int v^a|\nabla v|^2Q^{2p-4}.\nonumber
\end{align}
Recalling that $k=(\beta+1)(p-1)-\beta q$ and $R_{ij}=(n-1)g_{ij}$, combining \eqref{2.7}, \eqref{2.8} and \eqref{2.9} with \eqref{2.4},  we arrive at the following integral idendity.
\begin{proposition}
	If $v$ is a positive  solution for the equation \eqref{2.3}, then we  have 
\begin{equation}\label{2.10}
	\begin{aligned}
		&\left(-\beta q+\frac{-a+an^2\varepsilon^2}{n-1+n^2\varepsilon^2}\right)\int v^{a-1}Q^{p-2}|\nabla v|^2X^i_i+\frac{na}{n-1+n^2\varepsilon^2}\int v^{a-1}Q^{p-2}v_iv_j(Q^{p-2}v_i)_j\\+&
		\left[\frac{n(n-1)}{n-1+n^2\varepsilon^2}+\lambda(p-1-q)\right]\int v^a|\nabla v|^2Q^{2p-4}+\frac{n}{n-1+n^2\varepsilon^2}\int v^aF^i_jF^j_i\\-&(k+a)(\beta+1)(p-1)\int v^{a-2}|\nabla v|^4Q^{2p-4}-\beta\lambda(p-2)\int v^aQ^{2p-6}f=0.
	\end{aligned}
	\end{equation}
\end{proposition}
To address the last term for $p\neq 2$, we need the following lemma.
\begin{lemma}
	 We  have 
\begin{equation}\label{2.11}
		\begin{aligned}
		\int v^{a-1}Q^{p-2}|\nabla v|^2(Q^{p-2}v_i)_i&=-(a-1)\int v^{a-2}|\nabla v|^4Q^{2p-4}\\
	&	-\frac{p}{p-1}\int v^{a-1}(Q^{p-2}v_j)_iQ^{p-2}v_iv_j-\frac{p-2}{p-1}\alpha^2\int v^aQ^{2p-6}f. 
	\end{aligned}
	\end{equation}

\end{lemma}
\begin{proof} Combining 
	\begin{align*}
		L.H.S&=\int v^{a-1}Q^{p-2}v_jv_j(Q^{p-2}v_i)_i\\&=-(a-1)\int Q^{2p-4}v^{a-2}|\nabla v|^4-\int v^{a-1}(Q^{p-2}v_j)_iQ^{p-2}v_iv_j-\int v^{a-1}Q^{2p-4}v_iv_jv_{ij}
	\end{align*}
	and
\begin{align*}
	&(Q^{p-2}v_j)_iQ^{p-2}v_iv_j\\=&Q^{2p-4}v_{ji}v_iv_j+(p-2)Q^{2p-6}(\alpha^2v|\nabla v|^4+\beta^2|\nabla v|^2v_iv_kv_{ki})\\=& Q^{2p-4}v_{ji}v_iv_j-(p-2)\alpha^2vQ^{2p-6}f+(p-2)Q^{2p-6}v_iv_kv_{ik}(\alpha^2v^2+\beta^2|\nabla v|^2)\\=&(p-1)Q^{2p-4}v_{ji}v_iv_j-(p-2)\alpha^2vQ^{2p-6}f
\end{align*}
we get \eqref{2.11}.
\end{proof}
Therefore the last term in \eqref{2.10} becomes ,
\begin{equation}\label{2.12}
\begin{aligned}
	-\beta\lambda(p-2)\int v^aQ^{2p-6}f=&\frac{\beta\lambda(p-1)}{\alpha^2}\int v^{a-1}Q^{p-2}|\nabla v|^2X^i_i\\&+\frac{\beta\lambda(p-1)(a-1)}{\alpha^2}\int v^{a-2}|\nabla v|^4Q^{2p-4}\\&+\frac{\beta \lambda p}{\alpha^2}\int v^{a-1}(Q^{p-2}v_j)_iQ^{p-2}v_iv_j.
\end{aligned}
\end{equation}
It follows that we get the following important integral idendity. 
\begin{proposition}If $v$ is a positive solution for the equation \eqref{2.3}, then  for any  constants $\varepsilon, \beta, a$  we  have 
	\begin{equation}\label{2.12}
		\begin{aligned}
			0=&\left[-\beta q+\frac{-a+an^2\varepsilon^2}{n-1+\varepsilon^2n^2}+\frac{\beta \lambda(p-1)}{\alpha^2}\right]\int v^{a-1}Q^{p-2}|\nabla v|^2X^i_i\\&+\left[\frac{n(n-1)}{n-1+n^2\varepsilon^2}+\lambda (p-1-q)\right]\int v^a|\nabla v|^2Q^{2p-4}\\&+\left(\frac{\beta\lambda p}{\alpha^2}+\frac{na}{n-1+n^2\varepsilon^2}\right)\int v^{a-1}Q^{p-2}v_iv_j(Q^{p-2}v_i)_j\\&+\left[\frac{\beta\lambda(p-1)(a-1)}{\alpha^2}-(k+a)(\beta+1)(p-1)\right]\int v^{a-2}|\nabla v|^4Q^{2p-4}\\&+\frac{n}{n-1+n^2\varepsilon^2}\int v^aF^i_jF^j_i.
		\end{aligned}
	\end{equation}
\end{proposition}
\section{Proof of the Theorem\ref{Them1}}
In this section, through the choice of the constants $\varepsilon, \beta, a$,  we analyze the coefficients in \eqref{2.12},  and we  prove $|\nabla v|=0$, then $|\nabla \omega |=0$ so we complete the proof of our Theorem~\ref{Them1}.

Using \eqref{2.2} in the third term in \eqref{2.12}, we can rewrite \eqref{2.12}
\begin{equation}
		\begin{aligned}
			0=&\left[-\beta q+\frac{-a+an^2\varepsilon^2}{n-1+\varepsilon^2n^2}+\frac{\beta \lambda(p-1)}{\alpha^2}\right]\int v^{a-1}Q^{p-2}|\nabla v|^2X^i_i\\&+\left[\frac{n(n-1)}{n-1+n^2\varepsilon^2}+\lambda (p-1-q)\right]\int v^a|\nabla v|^2Q^{2p-4}\\&+\left(\frac{\beta\lambda p}{\alpha^2}+\frac{na}{n-1+n^2\varepsilon^2}\right)\int v^{a}F^i_jL^i_j\\&+\left(\frac{\beta\lambda p}{n\alpha^2}+\frac{a}{n-1+n^2\varepsilon^2}\right)\int v^{a-1}Q^{p-2}|\nabla v|^2X^i_i\\&+\left[\frac{\beta\lambda(p-1)(a-1)}{\alpha^2}-(k+a)(\beta+1)(p-1)\right]\int v^{a-2}|\nabla v|^4Q^{2p-4}\\&+\frac{n}{n-1+n^2\varepsilon^2}\int v^aF^i_jF^j_i.
		\end{aligned}
	\end{equation}
To be convenient, we let $$M=\frac{(\frac{\beta\lambda p}{\alpha^2}+\frac{na}{n-1+n^2\varepsilon^2})}{\frac{2n}{n-1+n^2\varepsilon^2}}.$$
	By \eqref{2.1b}, we get the following crucial integral idendity.
\begin{equation}\label{3.2}
	\begin{aligned}
		0=&\left[-\beta q+\frac{-a+an^2\varepsilon^2}{n-1+\varepsilon^2n^2}+\frac{\beta \lambda(p-1)}{\alpha^2}+\frac{\beta\lambda p}{n\alpha^2}+\frac{a}{n-1+n^2\varepsilon^2}\right]\int v^{a-1}Q^{p-2}|\nabla v|^2X^i_i\\
		&+\left[\frac{n(n-1)}{n-1+n^2\varepsilon^2}+\lambda (p-1-q)\right]\int v^a|\nabla v|^2Q^{2p-4}\\&+\frac{n}{n-1+n^2\varepsilon^2}\int v^a(F^i_j+ML^i_j)(F^j_i+ML^j_i)\\&+\Bigg[-\frac{1}{4}(\frac{\beta\lambda p}{\alpha^2}+\frac{na}{n-1+n^2\varepsilon^2})^2\frac{(n-1+n^2\varepsilon^2)}{n}\frac{(n-1)}{n} +\frac{\beta\lambda(p-1)(a-1)}{\alpha^2}\\&-(k+a)(\beta+1)(p-1)\Bigg]\int v^{a-2}|\nabla v|^4Q^{2p-4}.		\end{aligned}
\end{equation}	
Here we have only four terms and but three parameters $\beta, a, \varepsilon$, we shall choose them properly to cancel three terms.	

First, we choose $\varepsilon$  to make 
\begin{align}\label{3.3}
	\frac{n(n-1)}{n-1+n^2\varepsilon^2}+\lambda (p-1-q)=0,\end{align}	
from \eqref{3.3} we know the coffecient of the second term in  idendity \eqref{3.2} is zero.

To see this is possible, we show that 
\begin{lemma}
	If the constant $p, q, \alpha$ and $\lambda$ satisfy the condition in the Theorem~\ref{Them1}, then we have
	$n+\lambda(p-1-q)>0$.
	\end{lemma}	
\begin{proof}
	Recall that $\lambda=\alpha(n+1-\alpha q)$ and $\alpha=\frac{p}{q+1-p}$, then
	we need to show
	\begin{align*}
		n-p\left(n+1-\frac{pq}{q+1-p}\right)>0,
	\end{align*}
	which holds iff
	\begin{align*}
		\frac{pq}{q+1-p}>\frac{np+p-n}{p}.
	\end{align*}
The above inequality is reduced to 
	\begin{align*}
		(1-p)(n-p)q>(1-p)(np+p-n),
	\end{align*}
which is from the subcritical exponent of $q$, 
	\begin{align*}
		q<\frac{(n+1)(p-1)+1}{n-p}.
	\end{align*}
\end{proof}	
Now we take $\varepsilon=[\frac{n-\lambda(q+1-p)}{\lambda(q+1-p)}]^\frac{1}{2}(n-1)^{\frac{1}{2}}n^{-1}$ then we get
\begin{align}
	n^2\varepsilon^2=\frac{[n-\lambda(q+1-p)]}{\lambda(q+1-p)}(n-1), 
\end{align}
and 
\begin{align}
	\frac{n}{n-1+n^2\varepsilon^2}=\frac{\lambda(q+1-p)}{n-1}.
\end{align}
Second, we let $a=t\beta$, and take $t$ to make 
\begin{align}\label{3.5}
\frac{\lambda p}{n\alpha^2}-q+\frac{tn^2\varepsilon^2}{n-1+n^2\varepsilon^2}+\frac{\lambda(p-1)}{\alpha^2}=0.
\end{align}

From \eqref{3.5} we know the coffecient of the first term in the idendity \eqref{3.2} is zero.

By substituting $\varepsilon$, we take
\begin{align}\label{3.6}
	t=\left(q-\frac{\lambda(p-1)}{\alpha^2}-\frac{\lambda p}{n\alpha^2}\right)\frac{n}{n-\lambda(q+1-p)}.
\end{align}
Now we simplify it.
\begin{lemma} In fact 
	$t=\frac{n+1}{\alpha}$.
\end{lemma}
\begin{proof}
	First we have 
	\begin{align*}
		&n-\lambda(q+1-p)=\frac{1}{q+1-p}[(q+1-p)n-p(q+1-p)(n+1)+p^2q]\\=&\frac{1}{q+1-p}[qn+(1-p)n-(p-1)(q+1-p)(n+1)-(q+1-p)(n+1)+q+(p^2-1)q]\\=&\frac{p-1}{q+1-p}[1-(q+1-p)(n+1)+(p+1)q].
	\end{align*}
And we can get	\begin{align*}
		&n\left(q-\frac{\lambda(p-1)}{\alpha^2}-\frac{\lambda p}{n\alpha^2}\right)\\=&nq-\frac{n(n+1-\alpha q)(p-1)}{\alpha}-\frac{(n+1-\alpha q)p}{\alpha}\\=&nq-\frac{n+1-\alpha q}{\alpha}-\frac{n(n+1-\alpha q)(p-1)}{\alpha}-\frac{(n+1-\alpha q)(p-1)}{\alpha}\\=&\frac{nq\alpha-n-1+\alpha q}{\alpha}-\frac{(n+1-\alpha q)(p-1)(n+1)}{\alpha}\\=&\frac{(q\alpha-1)(n+1)}{\alpha}-\frac{(n+1-\alpha q)(p-1)(n+1)}{\alpha}\\=&\frac{n+1}{\alpha}[q\alpha-1-(n+1-\alpha q)(p-1)]\\
		=&\frac{n+1}{\alpha}[\frac{(q+1)(p-1)}{q+1-p}-(n+1-\alpha q)(p-1)]\\=&\frac{n+1}{\alpha}\frac{(p-1)}{(q+1-p)}[q+1-(n+1)(q+1-p)+pq].
	\end{align*}
\end{proof}

Now we substitute $\varepsilon$ and $a=\frac{n+1}{\alpha}\beta$ into the coefficient of $\int v^{a-2}|\nabla v|^4Q^{2p-4}$, and we shall find $\beta$ such that the coffecient of the last term in the idendity \eqref{3.2} is zero.

First we get the coefficient of $\int v^{a-2}|\nabla v|^4Q^{2p-4}$ in  \eqref{3.2}  is $g(\beta)$, where
\begin{align*}
	g(\beta)=&-\frac{1}{4}\left(\frac{\beta\lambda p}{\alpha^2}+\frac{na}{n-1+n^2\varepsilon^2}\right)^2\frac{(n-1+n^2\varepsilon^2)}{n}\frac{(n-1)}{n} +\frac{\beta\lambda(p-1)(a-1)}{\alpha^2}\\
	&-(k+a)(\beta+1)(p-1)\\
	=&-\frac{1}{4n}(\frac{\lambda p}{\alpha ^2}+\frac{n+1}{\alpha}\frac{\lambda(q+1-p)}{n-1})^2\frac{(n-1)^2}{\lambda (q+1-p)}\beta ^2+\frac{\lambda(p-1)}{\alpha^2}\frac{(n+1)}{\alpha}\beta^2 \\
	&-(p-1)^2\beta^2+q(p-1)\beta^2-\frac{(n+1)(p-1)}{\alpha}\beta^2\\&+q(p-1)\beta-\frac{(n+1)(p-1)}{\alpha}\beta-\frac{\lambda(p-1)}{\alpha^2}\beta-2(p-1)^2\beta \\&-(p-1)^2. 
	\end{align*}
We have the following lemma.
\begin{lemma}
	$\exists !\beta_0$ , such that $g(\beta_0)=0$.
\end{lemma}
\begin{proof}
	Since $g(\beta)$ is a quadratic function, we show that its determinant identically vanishes, which is
	\begin{equation}
	\begin{aligned}\label{3.7}
		&[-\frac{\lambda}{\alpha^2}-2(p-1)+q-\frac{n+1}{\alpha}]^2(p-1)^2\\+&4(p-1)^2[-\frac{1}{4n}(\frac{\lambda p}{\alpha ^2}+\frac{n+1}{\alpha}\frac{\lambda(q+1-p)}{n-1})^2\frac{(n-1)^2}{\lambda (q+1-p)}+\frac{\lambda(p-1)}{\alpha^2}\frac{n+1}{\alpha}\\&+q(p-1)-(p-1)^2-\frac{(n+1)(p-1)}{\alpha}]=0.	\end{aligned}
		\end{equation}
	We simplify it term by term, first we have 
	\begin{align*}
		&-\frac{\lambda}{\alpha^2}-2(p-1)+q-\frac{n+1}{\alpha}\\=&\frac{-(n+1)(q+1-p)+pq}{p}-2(p-1)+q-\frac{(n+1)(q+1-p)}{p}\\=&-\frac{2(n+1)(q+1-p)}{p}+2(q+1-p)\\=&\frac{2(p-n-1)(q+1-p)}{p}.
	\end{align*}
And we also have 
		\begin{align*}
			&\frac{\lambda p}{\alpha^2}+\frac{(n+1)}{\alpha}\frac{\lambda(q+1-p)}{n-1}\\=&(n+1-\alpha q)(q+1-p)[1+\frac{n+1}{n-1}]\\=&(n+1-\alpha q)(q+1-p)\frac{2n}{n-1}.
		\end{align*}
	We  simply the following term
		\begin{align*}
			&-\frac{1}{4n}(\frac{\lambda p}{\alpha ^2}+\frac{n+1}{\alpha}\frac{\lambda(q+1-p)}{n-1})^2\frac{(n-1)^2}{\lambda (q+1-p)}\\=&-\frac{1}{4n}(n+1-\alpha q)^2(q+1-p)^2\frac{4n^2}{(n-1)^2}\frac{(n-1)^2}{\lambda (q+1-p)}\\=&-\frac{n(n+1-\alpha q)(q+1-p)^2}{p}.		\end{align*}
	Then 	\begin{align*}
			\frac{\lambda(p-1)(n+1)}{\alpha^3}=\frac{(n+1-\alpha q)(p-1)(n+1)(q+1-p)^2}{p^2},
		\end{align*}
	and 	\begin{align*}
			-\frac{(n+1)(p-1)}{\alpha}=-\frac{(n+1)(p-1)(q+1-p)}{p}.
		\end{align*}
It follows that  its determinant 	\eqref{3.7} is equivalent to 
		\begin{align*}
			&\frac{(p-n-1)^2(q+1-p)^2}{p^2}-\frac{n(n+1-\alpha q)(q+1-p)^2}{p}\\&+\frac{(n+1-\alpha q)(p-1)(n+1)(q+1-p)^2}{p^2}\\&-\frac{(n+1)(p-1)(q+1-p)}{p}+(q+1-p)(p-1)=0.
		\end{align*}
		Multiplying $\frac{p^2}{q+1-p}$ and  using $(n+1-\alpha q)(q+1-p)=(q+1-p)(n+1)-pq $, it is equivalent for us to show 
		\begin{align*}
			&(p-n-1)^2(q+1-p)-pn(q+1-p)(n+1)+p^2n(q+1-p)+p^2n(p-1)\\+&(n+1)(p-1)(n+1)(q+1-p)-(n+1)(p-1)p(q+1-p)+(n+1)(p-1)p(1-p)\\-&p(p-1)(n+1)+p^2(p-1)=0,
		\end{align*}
		iff
		\begin{align*}
			(q+1-p)[(p-1-n)^2-pn(n+1)+p^2n+(n+1)^2(p-1)+(n+1)(p-1)p]=0,
		\end{align*}
		which is correct by direct computation.  
\end{proof}
From the expression $g(\beta)$, we can take $\beta_0$ such that $g(\beta_0)=0$. It follows that the first term, the second term and the last term in \eqref{3.2} is zero. At last we get the following result.
\begin{proposition}If $v$ is a positive  solution for the equation \eqref{2.3}, then  for the above determined  constants $\varepsilon, \beta, a$  we  have 
	\begin{align}\label{3.9}
		0=\frac{n}{n-1+n^2\varepsilon^2}\int v^a(F^i_j+ML^i_j)(F^j_i+ML^j_i).
	\end{align}
\end{proposition}
To deduce the desired results, we cite a key Lemma from \cite{AKM18} or \cite{CM18}.
\begin{lemma}
	Let the matrix A be symmetric with positive eigenvalues and let $\lambda_{min}$ and $\lambda_{max}$ be its smallest and largest eigenvalue, respectively; let B be a symmetric matrix, then
	\begin{align*}
		trace (AB(AB)^{T})\leq n(\frac{\lambda_{max}}{\lambda_{min}})^2trace((AB)^2).
	\end{align*}
\end{lemma} 
Now we show 
\begin{lemma}
	$F^i_j+ML^i_j=(AB)_{ij}$ where $A$, $B$ satisfy the conditions of the above Lemma.
\end{lemma}
	\begin{proof}From the definition of $F^i_j, \ L^i_j$ in the beginning of section 2, we have
		\begin{align*}
			F^i_j+ML^i_j&=(Q^{p-2}v_i)_j+\left(\varepsilon-\frac{1}{n}\right)X^l_lg_{ij}+M\frac{v_iv_j}{v}Q^{p-2}-\frac{M|\nabla v|^2}{nv}Q^{p-2}g_{ij}\\&=Q^{p-4}[(p-2)\beta^2v_lv_jv_{il}+(\alpha^2v^2+\beta^2|\nabla v|^2)v_{ij}]\\&+(p-2)Q^{p-4}\alpha^2vv_iv_j+\left(\varepsilon-\frac{1}{n}\right)X^l_lg_{ij}+M\frac{v_iv_j}{v}Q^{p-2}-\frac{M|\nabla v|^2}{nv}Q^{p-2}g_{ij}\\&=(N_1+N_2)_{ij},		\end{align*}
	where $(N_1)_{ij}=Q^{p-4}[(p-2)\beta^2v_lv_jv_{il}+(\alpha^2v^2+\beta^2|\nabla v|^2)v_{ij}]$.

	We rewrite
	\begin{align*}
		N_1=N_3N_4,
	\end{align*}
	where $(N_4)_{ij}=Q^{p-2}v_{ij}$, $(N_3)_{ij}=(p-2)\frac{\beta^2|\nabla v|^2}{\alpha^2 v^2+\beta^2 |\nabla v|^2}\frac{v_iv_j}{|\nabla v|^2}+\delta_{ij}$, $N_3$ is positive define with eigenvalues $1$ and $1+(p-2)\frac{\beta^2|\nabla v|^2}{\alpha^2v^2+\beta^2|\nabla v|^2}$.  From basic linear algebra we  have 
	\begin{align*}
		(N^{-1}_3)_{ij}=\delta_{ij}-(p-2)\frac{\beta^2|\nabla v|^2}{\alpha^2v^2+(p-1)\beta^2|\nabla v|^2}\frac{v_iv_j}{|\nabla v|^2}.
	\end{align*}
	Then
	\begin{align*}
		N_1+N_2=N_3(N_4+N_3^{-1}N_2).
	\end{align*}
	By direct calculations, $N^{-1}_3N_2$ is also a symmetric matrix.
	
	 Setting $A=N_3$, $B=N_4+N^{-1}_3N_2$ and we have done.	\end{proof}
 Now we prove the following last lemma.
	 \begin{lemma}
	 	$|\nabla v|=0$
	 \end{lemma}
	\begin{proof}
		By Lemma 3.5, Lemma 3.6, Proposition 3.4, we have
		\begin{align*}
			F^{i}_j+ML^i_j=0,
		\end{align*}
		which is
		\begin{align*}
			E^i_j+\varepsilon X^l_l g_{ij}+ML^i_j=0.
		\end{align*}
	
By taking trace we have 
\begin{align*}
	X^l_l=0.
\end{align*}
Then
\begin{align*}
	0=\int X^i_iX^j_j=\frac{n}{n-1}\int E^i_jE^j_i+\frac{n}{n-1}\int R_{ij}v_iv_jQ^{2p-4}.
\end{align*}
Following the method of Lemma 3.6, one can show that $\int E^i_jE^j_i\geq 0$, 
it forces that $|\nabla v|=0$.  Then we get $v$ is constant and complete the proof of Theorem\ref{Them1}.
\end{proof}
{\bf Acknowledgements:}
The second author was supported by  National Natural Science Foundation of China (grants 11721101 and 12141105) and National Key Research and Development Project (grants SQ2020YFA070080).

\end{document}